\documentclass[10pt]{amsart}
\usepackage{amssymb, amscd, amsmath, amsthm, 
 latexsym, enumerate}

\DeclareMathOperator{\cd}{cd}
\DeclareMathOperator{\hd}{hd}

\newcommand{\Z}{\mathbb Z}

\newtheorem{theorem}{Theorem}
\newtheorem{lemma}[theorem]{Lemma}
\newtheorem{cor}[theorem]{Corollary}

\begin{document}

\title{rank 1 abelian normal subgroups of 2-knot groups}

\author{Jonathan A. Hillman }
\address{School of Mathematics and Statistics\\
     University of Sydney, NSW 2006\\
      Australia }

\email{jonathan.hillman@sydney.edu.au}

\begin{abstract}
If the group of a 2-knot $K$ has an abelian normal subgroup of rank $\geq1$
which is not finitely generated then either $K$ is topologically equivalent to  Example 10 of
{\it ``A quick trip through knot theory"} or it has no minimal Seifert hypersurface.
\end{abstract}



\maketitle

If a 2-knot group $\pi$ has an abelian normal subgroup $A$ of rank $r\geq2$
then $\pi$ is a $PD_4$-group,  $r\leq4$ and $A\cong\Z^r$.
When $A$ has rank  $\leq1$ the picture is less clear.
If $r=1$ and the commutator subgroup $\pi'$ is infinite then $A$ is torsion-free.
If $A$ is not finitely generated then the Hirsch-Plotkin radical $\sqrt\pi$ is abelian of rank 1 
and $A\leq\sqrt\pi\leq\pi'$, 
and $\pi$ is either a $PD_4$-group or the solvable Baumslag-Solitar group $\Phi=BS(1,2)$
with presentation $\langle{t,a}\mid{tat^{-1}=a^2}\rangle$.
(See \cite[Chapter 15]{Hi} for these assertions.)
There are no known examples with $\pi'$ infinite and $A$ a non-trivial torsion group.
In contrast, any finitely generated abelian group is the centre of some 3-knot group
\cite{HK78}.

The most familiar constructions of 2-knots involve either (twist-)spinning,
surgery on finite sets of loops in connected sums of copies of $S^3\times{S^1}$
or surgery on a section of the mapping torus of a self-homeomorphism of a 3-manifold.
The associated knot groups either have deficiency 1 or have finitely generated commutator subgroup.
The group $\Phi$ is the group of Example 10 of Fox \cite{Fo},
and  has commutator subgroup $\Phi'\cong\mathbb{Z}[\frac12]$.
It is the only knot group arising from any of these constructions
which has an abelian normal subgroup which is not finitely generated.

It remains unknown whether there are any  2-knot groups other than $\Phi$
with an abelian normal subgroup which is not finitely generated.
If $\pi$ is such a group then we may ask whether $\sqrt\pi$ is finitely generated as a 
$\mathbb{Z}[\pi]$-module.
(Since $\sqrt\pi$ is torsion-free and abelian of rank 1,
 this is equivalent to it being the normal closure in $\pi$ of one element.)
Since $Aut(\sqrt\pi)$ is abelian, 
$\pi$ acts on $\sqrt\pi$ through $\pi/\pi'\cong\mathbb{Z}$,
and so the module structure is determined by the action of  a meridian.
If $\sqrt\pi$ is not finitely generated as a $\mathbb{Z}[\pi]$-module,
is it at least a minimax group?
Although we cannot yet answer these questions,
we can establish some restrictions.

We show first that $A$ is finitely generated as a $\Z[\pi]$-module if and
only if $\pi/A$ is finitely presentable.
Our main result is in \S2, 
where we show that Example 10 of \cite {Fo} is the only 2-knot (up to reflections)
with a minimal Seifert hypersurface  and with group of the type considered here.
(This is the main point at which we use the fact that we are working with 2-knots.
There are 2-knots which have no minimal Seifert hypersurface.
See \S5 of Chapter 17 of  \cite{Hi}.)
In \S3 we consider further the cases when $\pi$ is a $PD_4$-group.
We show that $\sqrt\pi/\sqrt\pi\cap\pi''$ is finite of odd order and
$\sqrt\pi$ meets nontrivially every subgroup which is not locally free.
In \S4 we assume that $A$ is an increasing union of infinite cyclic normal subgroups $C<\pi'$.
This includes the most interesting special case, 
when $\pi$ has infinite centre.
In the final section we comment briefly on torsion and on normal subgroups satisfying 
some finiteness condition.

\smallskip
\noindent{\it Acknowledgment.} I am grateful to T. Kanenobu for his advice on  \cite[Example 10]{Fo}.

\section{some basic observations}

The restrictions on abelian normal subgroups summarized in the first paragraph above
flow from more general results about closed 4-manifolds $M$ with $\chi(M)=0$,
rather than from specifically knot-theoretic arguments.
In \cite[Corollary 3.5.2]{Hi} it is shown that if $M$ is such a 4-manifold (or $PD_4$-complex) 
and $\pi=\pi_1(M)$ then $M$ is aspherical if $\beta_1^{2}(\pi)=0$ and $H^s(\pi;\Z[\pi])=0$ for $s\leq2$.
If $\pi$ has an infinite nilpotent normal subgroup $N$ of infinite index then $\beta_1^{2}(\pi)=0$ 
and $\pi$ has one end,  so $H^s(\pi;\Z[\pi])=0$ for $s\leq1$. 
If $N$ is central then it is easy to use standard LHS spectral sequence arguments to show that either 
$N$ is finitely generated or $H^2(\pi;\Z[\pi])=0$, in which case $\pi$ is a $PD_4$-group.
The subtleties arising in \cite[Theorem 15.10]{Hi} ensure that $H^2(\pi;\Z[\pi])=0$,
in situations where the finiteness assumptions needed to show that terms of an
LHS spectral sequence are 0 may not hold.
(The only knot-theoretic input needed for this theorem is that 
$\chi(M(K))=0$ and $\pi_1(M(K))=\pi{K}$ is not virtually $\Z^2$.)

if $G$ is a group then $G'$, $\zeta{G}$ and $\sqrt{G}$ are the commutator subgroup, 
the centre and the Hirsch-Plotkin radical of $G$, respectively.
(If $G$ is a 2-knot group then $\sqrt{G}$ is just the 
unique maximal nilpotent normal subgroup.)
A group is  {\it almost coherent\/} if every finitely generated subgroup is  $FP_2$.
We shall use \emph{knot group} as an abbreviation of \emph{high dimensional knot group}.

If $\pi$ is a  knot group then $\pi$ is a finitely presentable and $\pi/\pi'\cong\Z$,
and so $\pi$ is an HNN extension $HNN(B;\,\phi:J\to\widehat{J})$, 
with finitely generated base $B$ and associated subgroups  $J$ and $\widehat{J}$
 \cite{BS78}.
The choice of a meridian $t\in\pi$ determines an isomorphism
$\mathbb{Z}[\pi/\pi']\cong\Lambda=\mathbb{Z}[t,t^{-1}]$.

\begin{theorem}
Let $\pi$ be a knot group with a torsion-free abelian normal subgroup $A$ of rank $1$ such that $A\leq\pi'$.
Then $\pi$ acts on $A$ through $\pi/\pi'$,  so $A\leq\zeta\pi'$, 
and $\mathbb{Q}\otimes{A}\cong\mathbb{Q}[t,t^{-1}]/(pt-q)$ for some relatively prime $p,q$.
The subgroup $A$ is finitely generated as a $\Z[\pi]$-module if and only if $G=\pi/A$
is finitely presentable.
\end{theorem}

\begin{proof}
Since $A$ is torsion-free and of rank 1,
$Aut(A)$ is abelian, 
and so $\pi$ acts on $A$ through $\pi/\pi'$.
Hence $A\leq\zeta\pi'$.
If $t\in\pi$ is a meridian then $tat^{-1}=\frac{q}pa$, 
for some $\frac{q}p\in\mathbb{Q}^\times$ and all $a\in{A}$.
Clearly $\mathbb{Q}\otimes{A}\cong\mathbb{Q}[t,t^{-1}]/(pt-q)$.

Since $A<\pi'$, $G/G'\cong\pi/\pi'$ and so $\Z[G/G']\cong\Z[\pi/\pi']\cong\Lambda$.
The homology spectral sequence for $\pi$ as an extension of $G$ by $A$ 
gives an exact sequence
\[
H_2(\pi;\Lambda)\to{H_2(G;\Lambda)}\to\Lambda\otimes{A}\to{H_1(\pi;\Lambda)}\to{H_1(G;\Lambda)}\to0,
\]
since $H_i(A;\Z)=0$ for $i>1$.
If $G$ is finitely presentable then the $\Lambda$-modules $H_i(\pi;\Lambda)$
and $H_i(G;\Lambda)$ are finitely generated  for $i\leq2$,
since $\Lambda$ is a noetherian ring.
Hence $\Lambda\otimes{A}=A$ is also finitely generated as a $\Lambda$-module.

Conversely, if $A$ is finitely generated as a $\Z[\pi]$-module then it is cyclic,
as a module,  
since the subgroup generated by any finite subset of $A\setminus\{1\}$ is cyclic.
Hence $A$ is normally generated in $\pi$ by one element,
and so $G$ is finitely presentable.
\end{proof}

If $A$ is finitely generated as a $\Lambda$-module then $A\cong\Lambda/(pt-q)$.
Otherwise,  $A$ is divisible by arbitrarily large integers relatively prime to $pq$. 
Moreover, for any $a\not=1$ in $A$ the quotient 
$\pi/\langle\langle{a}\rangle\rangle$ 
is finitely presentable and has an infinite abelian normal subgroup, 
and so has one end.

\begin{cor}
If $\pi$ is almost coherent then we may assume that $p=1$, i.e., that
$\mathbb{Q}\otimes{A}\cong\mathbb{Q}[t,t^{-1}]/(t-q)$ for some $q\not=0$.
\end{cor}

\begin{proof}
Let $t$ be a meridian for $\pi$ and $x$ a nontrivial element of $A$.
Then the subgroup $B$ generated by $\{t,x\}$ has the presentation
\[
\langle{t,x}\mid {tx^pt^{-1}=x^q},~t^kxt^{-k}\leftrightharpoons{x},~\forall~k\rangle.
\]
Since $\pi$ is almost coherent, $B$ is $FP_2$, and so is
an HNN  extension with finitely generated base \cite{BS78}.
Since $B$ is solvable, the HNN extension must be ascending,
and so (after replacing $t$ by $t^{-1}$, if necessary) we may assume that $p=1$.
\end{proof}

There are uncountably many torsion-free abelian groups of rank 1.
Only countably many can occur as $\sqrt{G}$ for some finitely presentable group $G$.
A countable abelian group is the centre of some finitely presentable group if 
and only if it is recursively presentable \cite{OH07}.
What else can be said in the present context?
Must $\sqrt\pi$ be minimax?
(The minimax subgroups of $\mathbb{Q}$ are isomorphic to  $\mathbb{Z}[\frac1m]$,
for some $m\geq1$ - see \S7.4 of \cite{Bi}.
Every subgroup of $\mathbb{Q}$ which is not finitely generated
contains such a group,
for some $m>1$.)

The group $\Phi$  of  \cite[Example 10]{Fo}
has the Wirtinger presentation 
\[
\langle{t, u,v}\mid {vuv^{-1}=t},~tut^{-1}=v\rangle,
\]
which is realized by an immersion $R:D^3\looparrowright{S^4}$ with two ribbon singularities.
(\cite{Yaj69} -- see also  \cite[Chapter 1.7]{AIL}.)
The preimage of the double-point set consists of two properly embedded 2-discs,
which divide $D^3$ into the union of three portions $D_-$, $D^2\times[-1,1]$ and $D_+$,
and two 2-discs $d_-$ and $d_+$ in the interiors of $D_-$ and $D_+$, respectively.
The immersion $R$ is an embedding on the complement of these four 2-discs, 
while $R(d_-)=R(D^2\times\{1\})$ and  $R(d_+)=R(D^2\times{-1})$.
Let $N_\pm$ be a closed regular neighbourhood of $d_\pm$ in the interior of $D_\mp$. 
Then $V=\overline{D_-\setminus{N_-}}\cup_{\partial{N_-}}S^2\times[-1,1]\cup_{\partial{N_+}}
\overline{D_+\setminus{N_+}}$ is a 1-punctured copy of $S^2\times{S^1}$,
and $R|_{D\setminus{N_-\cup{N_+}}}$ extends to an embedding of $W$ into $S^4$,
with the image of $S^2\times[-1,1]$ being part of the boundary of a regular neighbourhood 
of $R(D^2\times[-1,1])$ in $S^4$.
Hence $\Phi$ is the group of a ribbon 2-knot $K=R|_{\partial{D^3}}$ with a Seifert hypersurface  
$V\cong\overline{S^1\times{S^2}\setminus{D^3}}$.
This is minimal, since $\pi_1(V)\cong\Z$ and $\Phi'\not=1$.

\section{HNN extensions}

Every knot group is an HNN extension with finitely generated base 
and associated subgroups.
We shall use a stronger assumption on the HNN structure.

A {\it Seifert hypersurface\/} for a 2-knot $K$ is a locally flat
codimension-1 submanifold 
$V\subset{X(K)}=\overline{S^4\setminus{K\times{D^2}}}$
such that $\partial{V}=K$.
It is {\it minimal\/} if the inclusion of $V$ into $X(K)$ is $\pi_1$-injective.
If $K$ has a minimal Seifert hypersurface then $\pi=\pi{K}$ is
an HNN extension $HNN(B,\varphi:J\cong\widehat{J})$, 
with base $B=\pi_1(X(K)\setminus{V})$ and associated subgroups $J,\widehat{J}\cong\pi_1(V)$.

\begin{theorem}
\label{HNNs}
Let $\pi=HNN(B;\varphi:J\cong\widehat{J})$ be a knot group which is an HNN extension 
with finitely generated base $B$ and associated subgroups $J,\widehat{J}$,
and such that every abelian normal subgroup of $I$ is finitely generated.
Suppose that $\pi$ has a torsion-free abelian normal subgroup $A$ of rank $1$.
If $J$ and $\widehat{J}$ are each proper subgroups of $B$ then $A\cong\mathbb{Z}$; 
otherwise $A\cong\mathbb{Z}[\frac1m]$, for some $m\geq1$.
\end{theorem} 

\begin{proof}
We may assume that $A\leq\pi'$, for otherwise $A\cong\mathbb{Z}$,
and we may take $m=1$.
Then $Aut(A)\leq\mathbb{Q}^\times$, and so $A\leq\zeta\pi'$.
Let $t$ be the stable letter of the HNN extension.
Let $B_{r,s}$ be the subgroup of $\pi'=\langle\langle{B}\rangle\rangle$
generated by $\cup_{r\leq{k}\leq{s}}t^kBt^{-k}$, for all $r\leq{s}$.
Then $B_{r,s}$ is an iterated generalized free product with amalgamation 
of copies of $B$, amalgamated over copies of $J\cong\widehat{J}$,
and $\pi'=\cup_{r\geq0}B_{-r,r}$.
In particular, $B_{q,r-1}\cap{B_{r,s}}\cong{I}$ and
$t^kB_{r,s}t^{-k}=B_{r+k,s+k}$, for all $q<r\leq{s}$ and $k$.
If $J$ and $\widehat{J}$  are each proper subgroups of $B$ then $\zeta\pi'\leq{I}$
\cite[Corollary 4.5]{MKS},
and so $\zeta\pi'$ and $A$ are finitely generated.
If, say, $B=J$ then $B\cap{A}\cong\mathbb{Z}$ 
and $t(B\cap{A})t^{-1}\leq{B\cap{A}}$.
Hence $tat^{-1}=a^m$ for some $m\not=0$  and all $a\in{I}\cap{A}$.
Since $A$ is torsion-free of rank 1, it follows that $tat^{-1}=a^m$ for all
$a\in{A}$.
Moreover,
$A=\cup_{k\geq0}t^{-k}(J\cap{A}t^k)$,
and so $A\cong\mathbb{Z}[\frac1m]$.
\end{proof}

The hypotheses of the first sentence hold if $\pi=\pi{K}$, 
where $K$ has a minimal Seifert hypersurface, 
or if $I$ is a $PD_3$-group.

\begin{cor}
Let $K$ be a $2$-knot with a minimal Seifert hypersurface and such that $\pi=\pi{K}$
has an abelian normal subgroup $A$ which is not finitely generated.
Then $\pi\cong\Phi$ and $K$ is TOP isotopic to  Example 10 of \cite{Fo} or to its reflection.
\end{cor}

\begin{proof}
The subgroup $A$ is torsion-free and of rank 1 \cite[Theorem 15.10]{Hi}.
Since $J=\pi_1(V)$ is a 3-manifold group, 
it is $FP_3$ and its abelian subgroups are finitely generated.
Since $A$ is not finitely generated, it is torsion-free and of rank 1.
Hence the second option must hold, and the HNN extension is ascending.
If $\pi$ is a $PD_4$-group and is an ascending HNN extension with $FP_3$ base $B$
then it follows from \cite[Lemma 3.4]{BG85} that $J=B=\pi'$.
But then $A$ must be finitely generated, contrary to our assumption.
Therefore $\pi\cong\Phi$.
Hence $K$ is TOP isotopic to Example 10 of \cite{Fo}, 
up to reflection \cite{Hi09}.
\end{proof}

This improves upon \cite[Theorem 15.16]{Hi}, 
where a similar argument was used to show that if a $2$-knot $K$ has a minimal Seifert hypersurface 
then the centre of $\pi{K}$ is one of the known examples $1,C_2,\Z, \Z\oplus{C_2}$ or $\Z^2$.
(All except for $C_2$ are realized by fibred 2-knots.
The group $C_2$ is realized by a satellite knot \cite{Yo82}.)

There are 2-knots with no minimal Seifert hypersurface, 
by the surgery argument of \cite[Chapter 17.\S5]{Hi}, but
it remains unknown whether every 2-knot group has an HNN decomposition 
with finitely presentable base and associated subgroups.

If a 2-knot group $\pi$ has two such abelian normal subgroups $A_1\not=A_2$ 
then $A_1A_2$ is a torsion-free,  nilpotent normal subgroup of Hirsch length $\leq2$, 
and so is abelian of rank 1 or 2.
In the latter case $A_1A_2\cong\mathbb{Z}^2$ \cite[Theorem 16.2]{Hi},
and the exponents $m_1$ and $m_2$ must be $\pm1$.
We may assume that $m_1=1$.
Most twist spins of torus knots give examples with $m_1=m_2=1$,
while 2-twist spins of certain Montesinos knots give examples with $m_1=1$ and $m_2=-1$.
(See \cite[Theorem 16.15]{Hi}.)

\section{abelian normal subgroups of $PD_4$-groups}

We assume henceforth that $\pi$ is a knot group which is a $PD_4$-group,
and which has an abelian normal subgroup of rank 1 which is not finitely generated.
(It is not known whether every knot group which is a $PD_4$-group
is the group of a 2-knot.)

\begin{lemma}
Let $G$ be an $FP_2$ group with one end and such that $\cd{G}=2$,
and let $A\cong\mathbb{Z}[\frac1m]$, where $m>1$. 
Then $\cd{A}\times{G}=4$.
\end{lemma}

\begin{proof}
Let $A_k=\frac1{m^k}\mathbb{Z}$.
Then $G_k=A_k\times{G}$ is $FP$, $\cd{G_k}=3$ 
and $H^s(G_k;F)=0$ for any free $\mathbb{Z}[G_k]$-module $F$ and $s<3$,
and for all $k\geq0$.
Moreover,  $G_k<{G_{k+1}}$ and $[G_{k+1}:G_k]=m<\infty$ for all $k$,
and $A\times{G}\cong\cup_{k\geq0}{G_k}$.
Hence $H^s(A\times{G};\mathcal{F})=0$ for any free $\mathbb{Z}[A\times{G}]$-module 
$\mathcal{F}$ and $s\leq3$ \cite[Theorem 3.1]{GS81}.
Since $\cd{A}\times{G}\leq\cd{A}+\cd{G}=4$, the lemma follows from \cite[Proposition 5.1(a)]{Bi}.
\end{proof}

Does the conclusion of this lemma hold if $FP_2$ is weakened to finitely generated?

\begin{theorem}
\label{fcodd}
Let $\pi$ be a knot group which is a $PD_4$-group with an abelian normal subgroup $A$ 
 which is not finitely generated.
Then $A\leq\pi'$ and $A/A\cap\pi''$ is finite cyclic of odd order.
\end{theorem}

\begin{proof}
Non-trivial subgroups of $\pi/\pi'\cong\Z$ are free abelian groups.
Since $A$ is torsion-free of rank 1 \cite[Theorem 15.10]{Hi}
but is not finitely generated it has no free direct summands, 
and so $A<\pi'$.

If $A\cap\pi''=1$ then $A$ embeds in the knot module $\pi'/\pi''$,
and so is a direct summand of a submodule of finite index.
Let $M$ be a complementary summand, and let $\sigma$ be the preimage of $M$ in $\pi$.
Then $A\sigma\cong{A}\times\sigma$, and so $\Z\times\sigma$ is a subgroup of $\pi'$.
Since $\pi$ is a $PD_4$-group and $\pi/\pi'\cong\mathbb{Z}$, 
$\hd\pi'\geq3$, while $\cd\pi'\leq3$ by \cite{St77}.
Hence $\hd\pi'=\cd\pi' =3$.
Since $\mathbb{Z}\times\sigma\leq\pi'$, we see that $\cd\sigma\leq2$.
Since $\hd{A}=1$ and $A\sigma$ has finite index in $\pi'$, 
we have  $1+\hd\sigma=\hd\pi'=3$,
and so $\hd\sigma=2$.
Therefore $\sigma$ is not locally free.
Let $\nu$ be a finitely generated subgroup of $\sigma$ which is not free,
and is indecomposable as a free product.
Then $\nu$ has one end and $\hd\nu=2$,
and so $\cd{A}\times\nu=4$, by Lemma 5.
But this contradicts $\cd\pi'=3$.
Hence $A/A\cap\pi''\not=1$.

Since $A/A\cap\pi''$ is a $\mathbb{Z}$-torsion submodule of the knot module $\pi'/\pi''$,
it is finite,  and since $A$ is torsion-free of rank 1 any finite quotient is cyclic.
It is of odd order, since $t$ and $t-1$ each act invertibly on $\pi'/\pi''$.
\end{proof}

If $\mathbb{Q}\otimes{A}\cong\mathbb{Q}[t,t^{-1}]/(pt-q)$ then $pt-q$ annihilates $A$,
and so the order of $A/A\cap\pi''$ is relatively prime to $p$, $q$ and $p-q$.

\begin{cor}
The characteristic class in $H^2(\pi'/A;A)$ for $\pi'$ as an extension 
of $\pi'/A$ by $A$ has infinite order. 
\end{cor}

\begin{proof}
The image of the characteristic class for the extension in $Hom(H_2(\pi'/A;\mathbb{Z}),A)$ 
is the connecting homomorphism 
$\delta$ in the five-term exact sequence of low degree
\begin{equation*}
\begin{CD}
H_2(\pi';\mathbb{Z})\to{H_2(\pi'/A;\mathbb{Z})}@>\delta>>{A}\to\pi'/\pi''\to\pi'/A\pi''\to0.
\end{CD}
\end{equation*}
(See \cite[Theorem 4]{Hi15}.)
The corollary follows since $A$ is infinite and torsion-free, and its image in $\pi'/\pi''$ is finite.
\end{proof}

In particular, $\pi'$ does not have a subgroup of finite index which splits as 
a direct product $A\times\sigma$.
In fact it is clear from the proof of Theorem \ref{fcodd} that no subgroup $\tau\leq\pi'$ 
which contains $A$ and such that $\hd\tau=3$ can split as such a direct product.

If $\pi$ is as in Theorem \ref{fcodd} and $\pi\not\cong\Phi$ 
then it is not elementary amenable \cite[Theorem 15.14]{Hi}.
Hence $\pi'/A$ is not locally finite.
Must it have an element of infinite order?
Since $\hd\pi'=\cd\pi'=3$, we may write $\pi'=\cup_{k\geq0}P_k$
as a union of finitely generated subgroups $P_k$ with $\hd{P_k}=\cd{P_k}=3$
and $A\cap{P_k}\not=1$.
(Hence $\zeta{P_k}$ is infinite.)

We may also show that $A$ meets non-trivially subgroups which are not locally free.

\begin{theorem}
Let $\pi$ be a knot group which is an almost coherent $PD_4$-group 
with an abelian normal subgroup $A$ which is not finitely generated.
Let $N$ be a subgroup of $\pi'$ which is not locally free. 
Then $A\cap\zeta{N}\not=1$.
\end{theorem}

\begin{proof}
Since $A$ is not finitely generated, it has rank 1 and $A\leq\pi'$.
Moreover, $A\leq\zeta\pi'$, by Theorem 1.
Hence if $A\cap{N}=1$ then $AN\cong{A}\times{N}$.
If $\cd{N}=3$ then $\cd\mathbb{Z}\times{N}=4$,
so ${A\times{N}}$ would have finite index in $\pi$,
and $A$ would be finitely generated.
Therefore we may assume that $\cd{N}=2$. 
Since $N$ is not locally free it has a finitely generated subgroup $\nu$ with one end.
Since $\pi$ is almost coherent, $\nu$ is $FP_2$.
Hence $\cd{A}\times\nu=4$, by Lemma 5.
Since $\cd{A}\nu\leq3$, we must have $A\cap{N}\not=1$,
and $A\cap{N}=A\cap\zeta{N}$, since $A\leq\zeta\pi'$.
\end{proof}

If $N$ is a locally free subgroup which is not abelian then $A\cap{N}=1$.
(However, $\Phi'$ is both abelian and locally free.)

\section{centres and infinite cyclic normal subgroups}

The group $A$ has infinite cyclic subgroups which are normal in $\pi$ if and only if $\frac{q}p=\pm1$. 
This includes the most interesting special case,
when $\pi$ is a $PD_4$-group and $\zeta\pi\not=1$.

The question as to whether centres of 2-knot groups are finitely generated
is closely related to the corresponding question for finitely generated groups of
cohomological dimension 3.
Consider the following assertions.

(A) {\sl If $G$ is finitely generated and $\cd{G}\leq3$ then $\zeta{G}$ is finitely generated};

(B) {\sl If $G$ is a $PD_4$-group then $\zeta{G}$ is finitely generated};

(C) {\sl If $G=\pi_1(M)$ where $M$ is a closed 4-manifold and $\chi(M)=0$ then
$\zeta{G}$ is finitely generated}.

\medskip
$A\Rightarrow{B}$. 
If $G$ is a $PD_4$-group and $\zeta{G}\not=1$ then $\chi(G)=0$.
We may assume that $G$ is orientable, and then there is an epimorphism $f:G\to\Z$.
If $f(\zeta{G})\not=0$ then $G$ is virtually $\Z\times\mathrm{Ker}(f)$,
and $\mathrm{Ker}(f)$ is $FP$ and has $\cd=3$.
Otherwise $G$ is an HNN extension $HNN(B,\varphi:I\cong{J})$ with $B$, $I$ and $J$ finitely generated,
and $\zeta{G}\leq{I}\cap{J}\leq{I}$.
(Moreover $\cd{I}=3$, by a Mayer-Vietoris argument.)

\medskip
$B\Rightarrow{C}$.
If $\zeta{G})$ has rank $\geq2$ then either $G$ is virtually $\Z^2$ or $H^2(G;\Z[G])=0$,
by LHS spectral sequence arguments.
In the latter case $M$ is aspherical \cite[Corollary 3.5.2]{Hi}.
If $\zeta{G}$ has rank 1 and is not finitely generated then the quotient $G/C$ must have one end,
for any $C\cong\Z<\zeta{G}$.
We again find that $H^2(G;\Z[G])=0$ and so $M$ is aspherical.
If $\zeta{G}$ is infinite then $\beta_1^{(2)}(G)=0$.
Hence if $\zeta{G}$ is a torsion group then it is finite \cite[Corollary 15.6.1]{Hi}.

In all cases $\zeta{G}$ is either free abelian or has rank 1.
This  is clear in cases B and C, 
and follows from \cite[Theorem 8.8]{Bi} if $\cd{G}\leq3$.
(The centres of finitely generated groups of cohomological dimension $\leq2$ 
are finitely generated \cite[Corollary 8.9]{Bi}.)

If $K$ is a 2-knot then $M(K)$ is a closed orientable 4-manifold with $\pi_1(M(K))\cong\pi{K}$ 
and $\chi(M(K))=0$, and so if $A$ holds then $\zeta\pi{K}$ is finitely generated.

\begin{lemma}
Let $\pi$ be a knot group which is a $PD_4$-group such that $\zeta\pi<\pi'$.
Then $\zeta\pi<J$, where $J$ is a finitely generated group
such that $\cd{J}=3$.
\end{lemma}

\begin{proof}
The group $\pi$ is an HNN extension $HNN(B;\,\phi:J\to\widehat{J})$, 
with $B,J$ and $\widehat{J}$ finitely generated.
Hence $\zeta\pi\leq\zeta{J}\cap\zeta\widehat{J}$, 
by consideration of normal forms in the HNN extension.
If $J=\zeta\pi$ then $\widehat{J}=J$ and $\langle{J,t}\rangle\cong{J}\times\Z$, 
and so $\pi\cong{B*_J(J\times\Z)}$.
Since $J\leq\pi'$, this is only possible if $J<B$, and then $J\times\Z$ has infinite index in $\pi$.
Hence $\cd(J\times\Z)\leq3$ \cite{St77},
and so $\cd{J}\leq2$.
But then $\cd\pi\leq3$, by a Mayer-Vietoris argument.
Since $\pi$ is a $PD_4$-group this is a contradiction.
Hence $\zeta\pi\not=J$.
\end{proof} 

\begin{lemma}
If a knot group $\pi$ has a normal subgroup $C\cong\mathbb{Z}$
such that $C<\pi'$ then either $C\leq\zeta\pi\cap\pi''$ 
or $[\pi:C_\pi(C)]\leq2$ and $C/C\cap\pi''$ is finite cyclic of odd order. 
\end{lemma}

\begin{proof}
This follows easily from the facts that $Aut(C)=\{\pm1\}$ and that $t-1$ acts
invertibly on $\pi'/\pi''$, since $\pi/\pi'\cong\Z$.
\end{proof}

For example, we may take the group $\pi$ with presentation
\[
\langle{t,a,b,c,d,e,f,z}\mid[a,b][c,d][e,f]=z^m,~az=za,~cz=zc,~ez=ze,
\]
\[
tat^{-1}=f,~tbt^{-1}=ef,~tct^{-1}=d,~tdt^{-1}=cd,~tet^{-1}=b, ~tft^{-1}=ab,
\]
\[tzt^{-1}=z^{-1}\rangle,
\]
with $m$ odd. 
It is easily seen that $\pi$ is the normal closure of $t$,
while $\sqrt\pi=\langle{z}\rangle$ and $\sqrt\pi/\sqrt\pi\cap\pi''\cong\mathbb{Z}/m\mathbb{Z}$.
This is the fundamental group of the mapping torus of a self-homeomorphism of
a $\widetilde{\mathbb{SL}}$-manifold, 
and surgery on a section of this mapping torus gives a 1-connected 4-manifold. 
The cocore of such a surgery is a 2-knot $K$, with knot group $\pi{K}\cong\pi$.
The knot manifold $M(K)$ is also the total space of an $S^1$-bundle 
over a non-orientable 3-manifold.
However $\sqrt\pi=\langle{z}\rangle\cong\Z$, 
and so $\pi$ has no abelian normal subgroup which is not finitely generated.

\begin{theorem}
Let $\pi$ be a knot group which is an almost coherent $PD_4$-group 
with an abelian normal subgroup $A$ which is not finitely generated.
If $\pi$ has a normal subgroup $N<\pi'$ and such that $\cd{N}=2$ then $[\pi:C_\pi(A)]\leq2$, 
$N'$ is free of rank $>1$ and $N$ is not finitely generated.
\end{theorem}

\begin{proof}
Since $\cd{N}=2$ and $N$ is not locally cyclic, 
either $N\cong\mathbb{Z}^2$ or $N'$ is free of rank $>0$ and 
$\zeta{N}\cong\mathbb{Z}$ \cite[Theorem 8.8]{Bi}.
In each case, $A\cap{N}\cong\mathbb{Z}$.
If $N$ is normal in $\pi$ then so is $A\cap{N}$, and so $\frac{q}p=\pm1$.
Hence $[\pi:C_\pi(A)]\leq2$.
Moreover $N'$ must then be free of rank $>1$, since $\sqrt{N}\leq\sqrt\pi=A$.

Suppose that $N$ is finitely generated.
Then $N/\zeta{N}$ is virtually free \cite[Theorem 8.4]{Bi},
and so has a characteristic subgroup of finite index 
which is a free group of finite rank.
Let $\overline{N}$ be the preimage of this subgroup in $\pi$.
Then  $\overline{N}$ is normal in $\pi$.
Since $\overline{N}\cong\Z\times{F(r)}$ for some $r\geq1$, 
it is a 2-dimensional duality group, with dualizing module 
$\mathcal{D}=H^2(\overline{N};\Z[\overline{N}])$.
Hence $H^2(\overline{N};\Z[\pi])\cong\mathcal{D}\otimes\Z[\pi/\overline{N}]$,
while $H^q(\overline{N};\Z[\pi])=0$ if $q\not=2$.

The LHS spectral sequence for $\pi$ (as an extension of $\pi/\overline{N}$ by $\overline{N}$,
with coefficients $\Z[\pi]$) collapses to give an isomorphism
\[
H^4(\pi;\mathbb{Z}[\pi])\cong
{H^2(\pi/\overline{N};H^2(\overline{N};\mathbb{Z}[\pi]))}\cong 
{H^2(\pi/\overline{N};\Z[\pi/\overline{N}])}\otimes\mathcal{D}.
\]
Since $\pi$ is a $PD_4$-group this is only possible if $\mathcal{D}\cong\Z$.
But then $r=1$, so $\overline{N}\cong\Z^2$ and hence $h(\sqrt\pi)\geq2$.
But $\sqrt\pi$ is torsion-free abelian of rank 1, by \cite[Theorem 15.7]{Hi}.
Therefore $N$ cannot be finitely generated.
\end{proof}

If $\pi$ is a $PD_4$-group and has a normal subgroup $C\cong\mathbb{Z}$ 
then $G=\pi/C$ is finitely presentable,
and an LHSSS argument shows that $H^3(G;\mathbb{Z}[G])\cong\mathbb{Z}$,
while  $H^q(G;\mathbb{Z}[G])=0$ if $q\not=3$.
These conditions may imply that $G$ is virtually a $PD_3$-group.
If so, it would follow that if $A$ is an abelian normal subgroup of rank 1
and $\pi$ acts on $A$ through $\pm1$ then $A$ must be finitely generated.
For we may assume that $A\leq\pi'$. 
Let $C<A$ be an infinite cyclic subgroup.
Then $A/C$ is an abelian normal torsion subgroup of $G=\pi/C$,
and so must be finite.

\begin{theorem}
Let $\pi$ be a knot group which is an almost coherent $PD_4$-group 
with an abelian normal subgroup $A$ which is not finitely generated.
If $[\pi:C_\pi(A)]$ is finite then $\pi''$ is not finitely generated.  
If $\pi$ is an ascending HNN extension then $[\pi:C_\pi(A)]=\infty$ and
$\pi''$ is not $FP_3$.
In either case,  $\cd\pi''=3$.
\end{theorem}

\begin{proof}
If $[\pi:C_\pi(A)]$ is finite then $\pi$ acts on $A$ through $\pm1$.
Hence $[\pi:C_\pi(A)]\leq2$,
and non-trivial elements of $A$ generate infinite cyclic subgroups $C<\pi'$ which are normal in $\pi$.
If $C$ is such a subgroup and $G=\pi/C$ then $G$ is finitely presentable
(in fact $FP_\infty$, since $\pi$ and $C$ are \cite[proposition 2.7]{Bi}).
A standard argument shows that $H^q(G;\Z[G])\cong{H^{q+1}(\pi;\Z[\pi])}$ for all $q$.
If $\widetilde{B}$ is a subgroup of $\pi'$ which contains $C$ and $B=\widetilde{B}/C$ 
then we see similarly that $H^q(B;\Z[G])=0$ for $q>2$,
since $\cd\widetilde{B}\leq3$.

The  Mayer-Vietoris sequence for cohomology with coefficients $\mathbb{Z}[G]$ 
associated to an HNN structure $G\cong{HNN(B;\,\phi:J\to\widehat{J})}$,
gives a short exact sequence
\[
0\to{H^2}(B;\mathbb{Z}[G])\to{H^2}(J;\mathbb{Z}[G])\to{H^3}(G;\mathbb{Z}[G])\cong\Z\to0,
\]
since $H^2(G;\Z[G])=H^3(B;\Z[G])=0$ and $H^3(G;\Z[G])\cong\Z$.
If the base $B$ is $FP_2$ then it follows from \cite[Lemma 3.4]{BG85}
(as in Theorem \ref{HNNs}) that the HNN extension cannot be properly ascending.
Since $C<\pi'$ and $\pi'$ is not finitely generated it follows that $\pi$ cannot be an
ascending HNN extension.

If $\pi''$ is finitely generated then $\pi/\pi''$ is finitely presentable,
and so is an HNN extension over a finitely generated base.
Since $\pi/\pi''$ is metabelian the extension is ascending.
The HNN structure lifts to an ascending HNN structure for $\pi$,
and the base $B$ is again finitely generated,
since $\pi''$ is finitely generated.
Hence $B$ is $FP_2$, since $\pi$ is almost coherent.
This contradicts the conclusion of the first paragraph, 
and so $\pi''$ is not finitely generated.

It follows also that if $\pi$ is a properly ascending HNN extension then $[\pi:C_\pi(A)]=\infty$.
If $\pi''$ were $FP_3$ then it would be $FP$.
Let $\tau$ be the preimage in $\pi$ of the torsion subgroup of $\pi'/\pi''$.
Then $\tau$ is also $FP$, since it is torsion-free and $[\tau:\pi'']$ is finite,
and $\pi/\tau$ is torsion-free metabelian of finite Hirsch length.
Hence $\cd\pi=\cd\tau+\cd\pi/\tau$ \cite[Theorem 5.5]{Bi},
and so $\cd\pi/\tau=1$.
Since $\pi/\tau$ is solvable, this is only possible if $\tau=\pi'$.
But $\pi'$ is not finitely generated, 
and so $\pi''$ cannot be $FP_3$.

Since $\pi$ is a $PD_4$-group, $\pi''$ is not abelian, for otherwise $\pi$
would be polycyclic and all subgroups would be finitely generated.
Since $[A:A\cap\pi'']$ is finite, by Theorem 6,
$A\cap\pi''$ is not finitely generated, and so $\cd(A\cap\pi'')=2$.
Since it is central in $\pi''$, 
it follows that $\cd\pi''>2$ \cite[Theorem 8.8]{Bi}.
Hence $\cd\pi''=\cd\pi'=3$.
\end{proof}

\section{torsion and finiteness conditions}

In this section we shall broaden the class of normal subgroups considered.
Since the 2-knot groups with finite commutator subgroup are known, 
we shall assume that $\pi$ is a 2-knot group such that $\pi'$ is infinite.

Let $E$ be the unique maximal normal elementary amenable subgroup of $\pi$.
Then $\sqrt\pi\leq{E}$.
If $E$ is infinite then $\pi$ has one end and $\beta_1^{(2)}(\pi)=0$, 
and so $\pi$ has no non-trivial finite normal subgroup \cite[Corollary 15.6.2]{Hi}.
If, moreover, $\pi$ is almost coherent and $h(E)>1$ then either $\pi\cong\Phi$ 
or $\pi$ is a $PD_4$-group \cite[Theorem 15.14]{Hi}. 
In the latter case, $E$ is torsion-free and polycyclic. 
(See also \cite[Theorem 9.1]{Hi}.)

\begin{lemma}
Let $\pi$ be a $2$-knot group with an infinite solvable normal subgroup $S$.
Then either $\pi$ has an infinite torsion-free abelian normal subgroup $A$ or
it has an elementary abelian normal subgroup $A\cong{C_p^\infty}$, for some prime $p$.
In either case we may assume that $\pi/A$ is finitely presentable.
\end{lemma}

\begin{proof}
We may assume that $\pi'$ is infinite, for otherwise $\pi$ is virtually $\mathbb{Z}$.
Let $S^{(n)}$ be the lowest non-trivial term of the derived series for $S$.
Then $S^{(n)}$ is abelian, and is normal in $\pi$,  since it is characteristic in $S$.
If $S^{(n)}$ is torsion-free let $A$ be the normal closure of a non-trivial element of $S^{(n)}$.
Then $A$ is an infinite torsion-free abelian normal subgroup.
If $S^{(n)}$ is not torsion-free then it has an element $a$ of prime order $p$,
for some prime $p$.
Let $A$ be the normal closure of $a$ in $\pi$.
Then $A$ is an elementary abelian $p$-group.
Since $\pi$ has an infinite solvable normal subgroup and $\pi'$ is infinite,
$\pi$ has one end and $\beta_1^{(2)}(\pi)=0$.
Therefore $\pi$ has no non-trivial finite normal subgroup \cite [Corollary 15.6.2]{Hi},
and so $A$ has countably infinite rank.
In either case $\pi/A$ is finitely presentable.
\end{proof}

Note that if $S^{(n)}$ has $p$- and $q$-torsion for at least two primes $p,q$ 
and $A\cong{C_p^\infty}$ is as in the lemma then $\pi/A$ has one end.

If $h(\sqrt\pi)=1$ and $\sqrt\pi$ is nilpotent then it is torsion-free and abelian of rank 1
\cite[Theorem 15.10]{Hi}.
Otherwise $\sqrt\pi$ may have non-trivial torsion.
The set of elements of finite order in a locally nilpotent group is a characteristic subgroup, 
with torsion-free quotient,
and is the direct product (over all primes $p$) of its $p$-primary subgroups,
which are again characteristic, but may not be nilpotent.

We know of no examples of this type, with non-trivial torsion,
but have not yet found an argument to exclude them.

Instead of considering normal subgroups which are abelian or locally nilpotent,
we could impose finiteness conditions.
Suppose that $\pi$ has a finitely generated infinite normal subgroup  $\nu$ of infinite index.
Then $\phi=\pi/\nu$ is finitely presentable,
$\pi$ has one end and $\beta_1^{(2)}(\pi)=0$.
There is an LHS spectral sequence
\[
E_2^{p,q}=H^p(\phi;H^q(\nu;\Z[\pi])\Rightarrow{H^{p+q}(\pi;\Z[\pi])},
\]
with differential $d_2$ of bidegree $(2,-1)$.
Since $\nu$ is infinite, $E_2^{p,0}=0$ for all $p$,
and since $\phi$ is finitely presentable,
$E_2^{p,1}\cong{H^p(\phi;\Z[\phi])}\otimes{H^1(\nu;\Z[\nu])}$ for $p\leq2$.
If $\nu$ is $FP_2$ and either $\nu$ or $\phi$ has one end then $H^2(\pi;\Z[\pi])=0$.
In this case $\pi$ is a $PD_4$-group.
Conversely, if $\pi$ is a $PD_4$-group then
$E_2^{1,1}\cong{H^1(\phi;\Z[\phi])}\otimes{H^1(\nu;\Z[\nu])}$ is 0,
since $H^2(\pi;\Z[\pi])=0$.
Hence either $\nu$ or $\phi$ has one end.

If we assume that $\pi$ is a $PD_4$-group then stronger finiteness conditions
on $\nu$ give finer detail.
If $\nu$ is $FP_2$ and $\phi$ has two ends then $\nu$ has finite index in $\pi'$,
since a group with two ends and cyclic abelianization maps onto $\Z$ with finite kernel.
Hence $\pi'$ is $FP_2$, and so $\pi'$ and $\nu$ are $FP_3$ and are $PD_3$-groups
\cite{HK07}. 

If $\nu$ is $FP_3$ then $\nu$ and $\phi$ are $FP_\infty$ \cite[Proposition 2.7]{Bi}.
Corner arguments then show that $\nu$ is a $PD_r$-group for some $r=1,2$ or 3,
while $H^s(\phi;\Z[\phi])=0$ for $s<4-r$ and $H^{4-r}(\phi;\Z[\phi])\cong\Z$.
If $r=1$ then $\nu\cong\Z$; if $r=2$ then $\phi$ is virtually a $PD_2$-group,
and so is an extension of $\pi^{orb}(B)$ by a finite normal subgroup,
for some aspherical 2-orbifold $B$; and if $r=3$ then $\phi$ has two ends.
If $r=2$ then $\pi$ is an extension of $\pi^{orb}(B)$
by a $PD_2$-group.
Twist spins of torus knots give examples with $\nu\cong\Z^2$ (except for a few ``small" cases").
If $r=2$ and $\chi(\nu)<0$ then $B$ must be one of the three flat orbifolds 
$S(2,3,6)$, $\mathbb{D}(3,\overline3)$ or $\mathbb{D}(\overline3, \overline3, \overline3)$,
since $\chi(\phi)=0$ and $\phi^{ab}$ is cyclic.
No such examples are known.
See \cite{Hi13} for further discussion of the possibilities with $\chi(\nu)<0$.


\end{document}